\documentclass[11pt]{amsart}
\usepackage[latin1]{inputenc}
\usepackage{pdfsync}
\usepackage[english]{babel}

\usepackage[colorlinks=true,urlcolor=blue,linkcolor=black,citecolor=black]{hyperref}
\usepackage{graphicx}

\usepackage[capitalize]{cleveref}
\usepackage{fullpage}
\usepackage{cite}
\usepackage{color}
\usepackage{amsmath}
\usepackage{amssymb}
\usepackage{amsfonts}
\usepackage{mathrsfs}
\usepackage{verbatim}
\usepackage{bbm}
\usepackage{enumerate}

\usepackage{pgfplots}
\pgfplotsset{compat=1.18}

\usepackage[left= 1.25 in, right= 1.25 in ,top= 1 in, bottom = 2 in]{geometry}
\newcommand{\R}{\mathbb{R}}

\newcommand{\C}{\mathbb{C}}

\newcommand{\N}{\mathbb{N}}

\renewcommand{\P}{\mathbb{P}}
\newcommand{\A}{\mathbb{A}}

\newcommand{\sse}{\subseteq}

\newcommand{\EC}[1]{\underset{n \in [N]\,}{\A^{#1}}}

\newcommand{\ECC}[2]{\underset{#1 \in #2\,}{\A}}
\newcommand{\ECw}{\underset{n \in [N]\,}{\A^w}}

\newcommand{\Elog}{\underset{n \in [N]\,}{\A^{\log}}}
\newcommand{\Elogg}[2]{\underset{#1 \in #2\,}{\A^{\log}}}

\newcommand{\Ellog}{\underset{n \in [N]\,}{\A^{\log \log}}}

\newcommand{\AVG}{\frac{1}{N} \sum_{n=1}^{N}}

\newcommand{\limAVG}{\lim_{N \to \infty} \frac{1}{N} \sum_{n=1}^{N}}

\newcommand{\norm}[1]{\left\lVert#1\right\rVert}
\newcommand{\floor}[1]{\lfloor #1 \rfloor}

\newtheorem{theorem}{Theorem}[section]
\newtheorem*{theorem*}{Theorem}
\newtheorem{prop}[theorem]{Proposition}
\newtheorem{lemma}[theorem]{Lemma}

\newtheorem{corollary}[theorem]{Corollary}
\newtheorem*{corollary*}{Corollary}

\theoremstyle{definition}
\newtheorem*{definition*}{Definition}
\newtheorem{defn}[theorem]{Definition}

\theoremstyle{remark}

\newtheorem{remark}[theorem]{Remark}
\newtheorem*{remark*}{Remark}

\makeatletter
\def\thmhead@plain#1#2#3{%
  \thmname{#1}\thmnumber{\@ifnotempty{#1}{ }\@upn{#2}}%
  \thmnote{ {\the\thm@notefont#3}}}
\let\thmhead\thmhead@plain
\makeatother
\theoremstyle{definition}

\title{Ergodic Averages Along Sequences of Slow Growth}

\author{Kaitlyn Loyd}
\address{Department of Mathematics\\
University of Maryland\\
College Park, MD 20742}
\email{\href{mailto:loydka@umd.edu}{loydka@umd.edu}}

\author{Sovanlal Mondal}
\address{Department of Mathematics\\
The Ohio State University \\
Columbus, OH 43210}
\email{\href{mailto:mondal.56@osu.edu}{mondal.56@osu.edu}}

\date{\today}

\thanks{The first author was partially supported by NSF grants DMS-1502632 and DMS-2402158.}

\begin{document}

\begin{abstract}
    We consider pointwise convergence of weighted ergodic averages along the sequence \( \Omega(n) \), where \( \Omega(n) \) denotes the number of prime factors of \( n\) counted with multiplicities. It was previously shown that \( \Omega(n) \) satisfies the strong sweeping out property, implying that a pointwise ergodic theorem does not hold for \( \Omega(n) \). We further classify the strength of non-convergence exhibited by \( \Omega(n) \) by verifying a double-logarithmic pointwise ergodic theorem along \( \Omega(n) \). In particular, this demonstrates that \( \Omega(n) \) is not inherently strong sweeping out. We also show that the strong sweeping out property for slow growing sequences persists under certain perturbations, yielding natural new examples of sequences with the strong sweeping out property. 
\end{abstract}

\maketitle

\section{Introduction}

    For \( n \in \N \), let \( \Omega(n) \) denote the number of prime factors of \(n \), counted with multiplicities. Though the asymptotic behavior of \( \Omega(n) \) has been long studied in analytic number theory, this paper continues the study of this sequence from a dynamical point of view. In particular, we provide a finer classification of the strength of non-convergence exhibited by the ergodic averages along \( \Omega(n) \). The dynamical approach to the study of \( \Omega(n) \) was introduced by Bergelson and Richter \cite{BR2020}, who prove that in uniquely ergodic systems \( (X, \mu, T )\), 
    \[
        \limAVG f(T^{\Omega(n)} x) = \int f \, d\mu
    \]
    for all continuous functions \( f \) and every \(x \in X\). The assumptions of unique ergodicity and continuity are key to the Bergelson-Richter Theorem. In \cite{Loyd2021}, it was shown that, given any non-atomic ergodic system, pointwise convergence for \( L^p \) functions fails to hold, \( 1 \leq p \leq \infty\). 
    
    \begin{theorem}[\cite{Loyd2021}]
    \label{thm: Pointwise Nonconvergence}
        Let \((X, \mathcal{B}, \mu, T)\) be a non-atomic ergodic dynamical system. Then for all \( \varepsilon > 0 \), there is a set \(A \in \mathcal{B}\) such that \( \mu(A) < \varepsilon \) and 
        \[
            \limsup_{N \to \infty} \AVG \mathbbm{1}_A(T^{\Omega(n)}x) = 1  
            \quad 
            \text{ and } 
            \quad
            \liminf_{N \to \infty} \AVG \mathbbm{1}_A(T^{\Omega(n)}x) = 0
        \]
        for \(\mu\)-almost every \(x \in X\), where \( \mathbbm{1}_A \) denotes the indicator function of a set \(A\).
    \end{theorem}

    \cref{thm: Pointwise Nonconvergence} verifies that the sequence \( \Omega(n) \) exhibits the \textit{strong sweeping-out property}, a stronger notion of non-convergence than non-existence of a pointwise ergodic theorem (see \cref{ss: SSO} for relevant definitions). Thus, sequences with the strong sweeping out property are considered particularly badly behaved for Ces\`aro averages. However, given a strong sweeping out sequence, it is possible there exists a form of averaging weaker than Ces\`aro which recovers pointwise convergence. For instance, this is true of \( \floor{\log n} \) and \( \floor{\log \log n} \). In \cref{sec: Proof of Double Log Convergence}, we show this to be the case for \( \Omega(n) \) by proving a double logarithmic pointwise ergodic theorem. 

    \begin{theorem}
    \label{thm: Double Log Convergence}
        Let \( (X, \mu, T) \) be an ergodic dynamical system. Then for all \( f \in L^1(\mu) \),
        \[
            \lim_{N \to \infty} \frac{1}{\log \log N} \sum_{n = 1}^N \frac{1}{n \log n} \, f(T^{\Omega(n)}x) 
            = 
            \int f \, d\mu
        \]
        for \( \mu \)-almost every \( x \in X \). 
    \end{theorem} 
    
    \cref{thm: Double Log Convergence} differentiates the ergodic behavior of \( \Omega(n) \) from strong sweeping out sequences such as \( 2^n \), for which, under mild conditions, no weaker form of averaging admits pointwise convergence \cite{RosenblattWierdl1994}. Such sequences are called \textit{inherently strong sweeping out}. Then, although \( \Omega(n) \) has the strong sweeping out property, it is not inherently strong sweeping out.   
    
    In \cref{sec: Proof Log SSO}, we demonstrate that \cref{thm: Double Log Convergence} is sharp, in the sense that \( \Omega(n) \) retains the strong sweeping out property for logarithmic averages.
    
    \begin{theorem}
    \label{thm: Log Sweeping Out}
        Let \((X, \mathcal{B}, \mu, T)\) be a non-atomic ergodic dynamical system. Then for all \( \varepsilon>0 \), there is a set \(A \in \mathcal{B}\) such that \( \mu(A) < \varepsilon\) and
            \[
                \limsup_{N \to \infty} \frac{1}{\log N} \sum_{n = 1}^N \frac{\mathbbm{1}_A(T^{\Omega(n)}x)}{n}  = 1  
                \quad 
                \text{ and } 
                \quad
                \liminf_{N \to \infty}\frac{1}{\log N} \sum_{n = 1}^N \frac{\mathbbm{1}_A(T^{\Omega(n)}x)}{n}  = 0
            \]
        for \( \mu \)-almost every \(x \in X\).
    \end{theorem}

    Proving \cref{thm: Double Log Convergence,thm: Log Sweeping Out} utilizes input from both the number theoretic and dynamical points of view. Taking advantage of the repetition of \( \Omega(n) \), we convert the ergodic averages to weighted averages of the form
    \[
        \frac{1}{\sum_{k \leq N} w_N(k)} \sum_{k \leq N} w_N(k) f(T^k x) .
    \]
    Understanding these weight functions is closely tied to analyzing the distribution of the set of \(k \)-almost primes, those integers satisfying \( \Omega(n) = k \) for fixed \( k \) (see \cref{ss: NT Weights} for more details). Moreover, it is key to the proof of \cref{thm: Log Sweeping Out} that the corresponding weight functions are largely supported in relatively small intervals. Specifically, the weight functions satisfy
    \begin{equation}
    \label{eqn: Weight Support}
        \limsup_{N \to \infty} \frac{1}{\sum_{k \leq N} w_N(k)} \sum_{k \notin I_N} w_N(k) 
        \leq 
        \varepsilon
    \end{equation}
    for all \( \varepsilon > 0\), where \( I_N \) is an interval centered at \(\log\log N\) with length of order \( \sqrt{\log\log N}\). We leverage the support and shape of the weight functions to employ a dynamical argument due to del Junco and Rosenblatt \cite{dJR1979}. 

    In contrast, the weight functions appearing in the proof of \cref{thm: Double Log Convergence}, denoted \( \eta_N(k) \), fail to satisfy a condition as strong as Equation \eqref{eqn: Weight Support}. However, expanding the interval of support, a weaker estimate holds:
    \begin{equation}
    \label{eqn: eta weight support}
        \limsup_{N \to \infty} \frac{1}{\log \log N} \sum_{k \geq L_N} \eta_N(k)
        \leq 
        \varepsilon
    \end{equation}
    for all \( \varepsilon > 0\), where \( L_N \) denotes the right-hand endpoint of the interval \( I_N\). To prove \cref{thm: Double Log Convergence}, we first verify pointwise convergence for bounded functions using Equation \eqref{eqn: eta weight support} in combination with an approximation of the weight functions due to Gorodetsky, Lichtman, and Wong \cite{GLW2023}. We then prove the following maximal ergodic theorem for the desired weighted averages. 
    
    \begin{prop}
    \label{prop: MEI for MPS}
        There exists a \( C> 0 \) such that for any measure-preserving dynamical system \( (X, \mu, T) \), \( f \in L^1(\mu)\), and \( \lambda >0\), 
        \[
            \mu \Big\{ x \in X \,:\,  \sup_{N\in \N} \Big| \frac{1}{\log\log N} \sum_{k \leq L_N} \eta_N(k) f(T^{k}x) \Big| > \lambda \Big\} \leq \frac{C}{\lambda} \norm{f}_1. 
        \]
    \end{prop}

    In \cref{sec: Proof Perturbations SSO}, we prove that the strong sweeping out property for slow growing sequences persists under perturbations.  

    \begin{theorem}
    \label{thm: Perturbations of SSO}
        Let \((b_n) \sse \N\) be monotonically increasing to infinity. Let \((a_n)\) be a perturbation of the sequence \((b_n)\), in the following sense. Define \( p_n := \max_{j \leq n} |a_j - b_j|\). Suppose
        \begin{enumerate}
            \item \(b_n = o(n^s) \) for all \( s > 0 \); 
            \item \( p_n = o(b_n) \).
        \end{enumerate}
        Then \( (a_n) \) has the strong sweeping out property.
    \end{theorem}

    \begin{remark}
        The proof of \cref{thm: Perturbations of SSO} demonstrates that the definition of \( p_n \) can be weakened to taking the maximum over integers excluding a fixed set of small natural density. 
    \end{remark}
    
    Increasing sequences of subpolynomial growth, i.e. those satisfying condition (1), are already known to be strong sweeping out (see for instance \cite{JW1994} or \cref{prop: JW SSO Condition}). \cref{thm: Perturbations of SSO} generalizes this result to include, not necessarily increasing, perturbations of such sequences. In particular, \cref{thm: Perturbations of SSO} implies \cref{thm: Pointwise Nonconvergence}. The classical Erd\H{o}s-Kac Theorem can be thought of as a central limit theorem for \( \Omega(n) \), stating, roughly speaking, that the number of prime factors of an integer \(n \leq N\) becomes normally distributed with mean and variance \( \log \log n \). In the context of \cref{thm: Perturbations of SSO}, the Erd\H{o}s-Kac Theorem implies that \( \Omega(n) \) is a sufficiently small perturbation of the strong sweeping out sequence \( \floor{\log \log n }\).

    \cref{thm: Perturbations of SSO} also widens the known class of strong sweeping out sequences. Many additive functions are known to possess a central limit theorem (see for instance \cite{Billingsley74}). \cref{thm: Perturbations of SSO} guarantees that those with subpolynomial mean and small enough variance have the strong sweeping out property. In fact, the following weaker condition on the normal of order of the sequence is sufficient to guarantee the strong sweeping out property. 

    \begin{corollary}
    \label{cor: CLT SSO}
        Let \( a(n) \sse \N \). Suppose there exist increasing subpolynomial functions \( b(n) \) and \( p(n) \) satisfying \(p(n) = o(b(n))\) and 
        \[
            \limsup_{N \to \infty} \frac{1}{N} \#\big\{ n \leq N \,:\, |a(n) - b(n)| > p(n)\big\} < \varepsilon
        \]
        for all \( \varepsilon > 0\). Then \( a(n) \) satisfies the strong sweeping out property.
    \end{corollary}

    \begin{corollary}
    \label{cor: Divisor and little omega}
        Let \( \omega(n) \) denote the number of distinct prime factors of \( n\) and \( d(n) \) the number of divisors of \( n\). Then the sequences \( \omega(n) \) and \( \log d(n) \) have the strong sweeping out property. 
    \end{corollary}

\section{Background and Notation}
\label{sec: Background}
    \subsection{Measure-preserving systems}
    \label{ss: MPS}
        By a \textit{topological dynamical system}, we mean a pair \((X,T)\), where \(X\) is a compact metric space and \(T\) a homeomorphism of \(X\). A Borel probability measure \(\mu\) on \(X\) is called \textit{T-invariant} if \(\mu(T^{-1}A) = \mu(A)\) for all measurable sets \(A\). By the Bogolyubov-Krylov theorem (see for instance \cite[Corollary 6.9.1]{Walters1982}), every topological dynamical system has at least one \(T\)-invariant measure. When a topological system \((X, T)\) admits only one such measure, \((X, T)\) is called \textit{uniquely ergodic}. 
    
        By a \textit{measure-preserving dynamical system}, we mean a probability space \((X, \mathcal{B}, \mu)\), where \(X\) is a compact metric space and \(\mathcal{B}\) the Borel \(\sigma\)-algebra on \(X\), and \( \mu \) a Borel probability measure, accompanied by a measure-preserving transformation \(T: X \to X\). We often omit the \(\sigma\)-algebra \(\mathcal{B}\) when there is no ambiguity. A measure-preserving dynamical system is called \textit{ergodic} if for any \(A \in \mathcal{B}\) such that \(T^{-1}A = A\), one has \(\mu(A) = 0\) or \(\mu(A) = 1\).   
        
        One of the most fundamental results in ergodic theory is the Birkhoff Pointwise Ergodic Theorem, which states that for any ergodic system \((X, \mu, T)\) and \(f \in L^1(\mu)\),
        \[
            \lim_{N \to \infty} \AVG f(T^n x) = \int f \, d\mu
        \]  
        for \(\mu\)-almost every \(x \in X\).

    \subsection{Weighted Averaging}
    \label{subsec: Weighted Averages}
        For \( N \in \N \), let \( [N] \) denote the set \( \{1, \dots, N\} \). Given \( a: \N \to \C \) and \(w: \N \to \C \), define the \textit{Ces\`aro averages of} \( a\) by
        \[
            \EC{} a(n) := \AVG a(n)
        \]
        and \(w\)-\textit{weighted averages of} \( a\) by
        \[
            \ECw a(n) := \frac{1}{W(N)} \sum_{n \leq N} w(n) a(n),
        \] 
        where \( W(N) := \sum_{n \leq N} w(n) \). Most frequently, we take \( w(n) = 1/n \) or \( 1/n \log n \), in which case \( W(N) = \log N \) and \( \log \log N \), respectively. The former are called the \textit{logarithmic averages of} \(a\), denoted \( \Elog a(n) \), and the latter the \textit{double-logarithmic averages of} \(a\), denoted \( \Ellog a(n) \). Despite the slight abuse of notation, whether the superscript is taken to be \(w\) or \(W\) is clear from context. 

        Under mild assumptions, convergence of Ces\`aro averages implies the convergence of other weighted averages. In this sense, we consider Ces\`aro convergence to be a strong notion of convergence. 
        
        \begin{lemma}
        \label{lem: Weaker Weights Proof}
            Let \( w: \N \to \C \) be a weight function satisfying
            \begin{enumerate}
                \item \( w(n) \) is non-negative and non-increasing;
                \item \( \lim_{N \to \infty} W(N) = \infty \),
            \end{enumerate}
            where \( W(n) = \sum_{n \leq N} w(n) \). If \( a: \N \to \C \) satisfies
            \(
                \lim_{N\to\infty} \EC{}a(n) = L
            \)
            for some \( L \in \C\), then
            \(
                \lim_{N \to \infty} \ECw \, a(n) = L . 
            \)
        \end{lemma}
        \begin{proof}
            Let \(\varepsilon > 0\). By a telescoping argument,  
            \begin{align}
            \label{eqn: Telescoping Weights1}
                \notag 
                \ECw a(n)  
                &= 
                \frac{1}{W(N)} \sum_{n \leq N} w(n) a(n) \\ 
                &= 
                \notag 
                \frac{1}{W(N)} \Bigg[ w(N) \sum_{n \leq N} a(n) +  \sum_{m \leq N-1} \big(w(m) - w(m+1)\big) \sum_{n \leq m} a(n) \Bigg]
                \\
                &=
                \frac{1}{W(N)} \Bigg[ N w(N) \Bigg( \frac{1}{N}\sum_{n \leq N} a(n) \Bigg) + \sum_{m \leq N-1} m \big(w(m) - w(m+1)\big) \Bigg(\frac{1}{m}\sum_{n \leq m} a(n) \Bigg) \Bigg].
            \end{align}
            Set \( \widetilde{w}(n) = n \big( w(n) - w(n+1) \big) \) for \( n \leq N-1 \) and \( \widetilde{w}(N) = N w(N) \). Notice that 
            \begin{equation}
            \label{eqn: Telescoping Weights2}
                W(N) = \sum_{n \leq N} w(n) = \sum_{n \leq N} \widetilde{w}(n).
            \end{equation}
            
            By assumption, there is some \( N_0 \) such that for \( m \geq N_0 \), 
            
            \begin{equation*}
                \Big| \frac{1}{m} \sum_{n \leq m} a(n) - L \Big| \leq \varepsilon.
            \end{equation*}
            Hence Equation \eqref{eqn: Telescoping Weights1} yields
            \begin{align*}
                \Big| \ECw a(n) - L \Big| 
                &\leq 
                \frac{\widetilde{w}(N)}{W(N)} \Big | \AVG a(n) - L \Big | 
                +  
                \frac{1}{W(N)}  \sum_{m \leq N-1} \widetilde{w}(m) \Big | \frac{1}{m}\sum_{n \leq m} a(n) - L \Big|
                \\
                &\leq
                \frac{\varepsilon \widetilde{w}(N)}{W(N)}
                + 
                \frac{\varepsilon }{W(N)} \sum_{m \geq N_0}  \widetilde{w}(m) 
                + 
                \frac{1}{W(N)} \sum_{m \leq N_0-1} \widetilde{w}(m) \Big|\frac{1}{m}\sum_{n \leq m} a(n) - L \Big|
                \\
                & \leq
                \frac{\varepsilon}{W(N)} \sum_{m \leq N} \widetilde{w}(m) + \frac{C_{\varepsilon}}{W(N)}
                \\
                & \leq 
                \varepsilon + \frac{C_\varepsilon}{W(N)},
            \end{align*}
            where the last inequality follows from Equation \eqref{eqn: Telescoping Weights2}. Taking \( N \) \large large enough so that \(\frac{C_\varepsilon}{W(N)} < \varepsilon\), we are done.
        \end{proof}

    \subsection{The Strong Sweeping Out Property}
    \label{ss: SSO}
        We introduce definitions, examples, and techniques related to the strong sweeping out property.

        \begin{defn}
            A sequence \( \{a_n\} \subset \N \) is called \textit{universally bad for} \(L^p\) if for any non-atomic ergodic system \( (X, \mu, T) \) and \( f \in L^p(\mu) \), the averages 
            \(
               \EC{} f(T^{a_n} x)
            \)
            diverge almost everywhere. 
        \end{defn}

        Krengel \cite{Krengel1971} proved the existence of a strictly increasing subsequence \((k_n) \) that is universally bad for \( L^p\), \( 1 \leq p \leq \infty \). In fact, he proves that \((k_n)\) satisfies a stronger notion of non-convergence, the strong sweeping out property. 

        \begin{defn}
            Let \( w: \N \to \C \) be a weight function. A sequence \( (a_n) \) has the \textit{strong sweeping out property for w-averages} if for any non-atomic ergodic system \( (X, \mathcal{B}, \mu, T) \) and \( \varepsilon > 0\), there is an \( A \in \mathcal{B} \) satisfying \( \mu(A) < \varepsilon\) and
            \[
                \limsup_{N \to \infty} \ECw \mathbbm{1}_A(T^{a_n}x) = 1
                \quad\quad \text{ and } \quad\quad 
                \liminf_{N \to \infty} \ECw \mathbbm{1}_A(T^{a_n} x) = 0
            \]
            for \(\mu\)-almost every \( x \in X\). When \( (a_n) \) is strong sweeping out for Ces\`aro averages, we simply say \( (a_n) \) has the \textit{strong sweeping out property}, or \( (a_n) \) is \textit{strong sweeping out}. 
        \end{defn} 

        Examples of strong sweeping out sequences include all lacunary sequences \cite{Akcetal1996}. Lacunary sequences are those satisfying \( a_{n+1}/a_n \geq 1 + \eta \) for some fixed \( \eta > 0\). Weakening this growth condition, sequences satisfying 
        \[
            \frac{a_{n+1}}{a_n} \geq 1 + \frac{1}{(\log \log n)^{1-\eta}}
        \]
        are also known to be strong sweeping out \cite{MRW2023}.

        Jones and Wierdl \cite[Theorem 2.16]{JW1994} give a general criterion for determining whether an increasing sequence has the strong sweeping out property (for Ces\`aro averages). 
        
        \begin{prop}[\cite{JW1994}]
        \label{prop: JW SSO Condition}
            Let \( (a_n) \) be an increasing sequence of positive integers. For \( 0 < \varepsilon < 1 \) , and \( 0 < p < q \), define 
            \[
            \varphi_\varepsilon(p, q) = \max_{p \leq n \leq q} \{a_n - a_{\floor{\varepsilon n}} \}.
            \]
            Suppose that for every \( 0 < \varepsilon < 1 \), \( u > 0 \), and \( C > 0 \), there are \( p = p(\varepsilon, u, C) \) and \( q = q(\varepsilon, u, C) \) such that \( u \leq p \), \( u < q \), and
            \[
                \frac{a_q - q_p}{\varphi_\varepsilon(p,q)} > C.
            \]
            Then \( (a_n) \) has the strong sweeping out property.
        \end{prop}

        \cref{prop: JW SSO Condition} yields numerous examples of slow growing sequences with the strong sweeping out property, including \( \floor{(\log n)^c}\) for \( c > 0\) and \( \floor{\log \log n}\). 

        Though many universally bad sequences  were later shown to be strong sweeping out, this is not always the case, demonstrating that the strong sweeping out property is a strictly stronger notion of non-convergence.  
        
        \begin{defn}
            A sequence \( (a_n) \sse \N \) is called \(\delta\)-\textit{sweeping out} if for any non-atomic ergodic system \( (X, \mathcal{B}, \mu, T) \) and \( \varepsilon > 0\), there is an \( A \in \mathcal{B} \) satisfying \( \mu(A) < \varepsilon\) and
            \[
                \limsup_{N \to \infty} \EC{} \mathbbm{1}_A(T^{a_n} x) \geq \delta
            \]
            for \(\mu\)-almost every \( x \in X\). 
        \end{defn}

        If a sequence is \(\delta\)-sweeping out for some \( \delta > 0\), then the sequence is universally bad for \(L^\infty\). However, it is possible for a sequence to be \(\delta\)-sweeping out for some \( \delta < 1\), yet not strong sweeping out (see for instance \cite{Akcetal1996}).

        Among strong sweeping out sequences, there is a further classification of the strength of non-convergence behavior. For certain strong sweeping out sequences, such as \( \floor{\log n}\), there exist weaker forms of averaging for which pointwise convergence does hold. However, some sequences are so badly behaved no weaker form of averaging can counteract this behavior. 

        \begin{defn}
            A sequence \( (a_n) \) is called \textit{inherently strong sweeping out} if \((a_n)\) is strong sweeping out for all non-negative, non-increasing weight functions \( w: \N \to \C \) satisfying \( \sum_{n \leq N} w(n) \to \infty \).
        \end{defn}

        The class of inherently strong sweeping out sequences includes all lacunary sequences for instance, as shown in \cite{RosenblattWierdl1994}. On the other hand, certain sequences are known to exhibit particularly good behavior. 

        \begin{defn}
            A sequence \( (a_n) \sse \N \) is called \textit{universally good for \( L^p \)} if for all measure preserving systems \( (X, \mu, T) \) and \( f \in L^p(\mu)\), the averages
            \(
                \EC{} f(T^{a_n} x) 
            \)
            converge \(\mu-\)almost everywhere. 
        \end{defn}
    
        In this language, Birkhoff's Pointwise Ergodic Theorem states that the integers are universally good for \( L^1 \). \cref{thm: Double Log Convergence} demonstrates that, when considering double-logarithmic averages, \( \Omega(n) \) is universally good for \( L^1 \). This is also known for certain sparse subsequences. Bourgain \cite{Bourgain_1988} showed that the sequence of perfect squares, and more generally any polynomial subsequence, is universally good for \( L^p \) for \( p > 1 \). Further, if \( p_n \) denotes the \(n\)-th prime integer, Wierdl \cite{Wierdl1988} proved that \( (p_n) \) is universally good for \( L^p \) for \( p>1\).

        At the heart of many strong sweeping out results is the failure of a maximal ergodic inequality. The following lemma allows us to disprove such an inequality for a chosen system and transfer the result to an arbitrary system.

        \begin{lemma}
        \label{lem: Transference}
            Let \((a_n)\) be a sequence of integers. Suppose that for any \(\varepsilon\in (0,1)\) and \(N_1>0\), there exist \(N_2 > N_1\) and a measure-preserving system \((\widetilde{X}, \widetilde{\mathcal{B}}, \nu, S)\) with a set \(\widetilde{E}\in \widetilde{\mathcal{B}}\) such that \(\nu (\widetilde{E})<\varepsilon\) and
            \begin{equation*}\label{eq:2.10} 
                \nu \Big\{ x \in \widetilde{X} \,:\, \sup_{N_1 \leq N \leq N_2} \EC{} \mathbbm{1}_{\widetilde{E}}(S^{a_n} x)
                > 1-\varepsilon \Big\}
                > 
                1 - \varepsilon.
            \end{equation*}
            Then in any non-atomic ergodic system \((X,\mathcal{B},\mu,T)\), there exists a set \(E\) with \(\mu(E)<\varepsilon\) such that
            \begin{equation*}
                \mu \Big\{x \in X : \sup_{N_1 \leq N \leq N_2} \EC{} \mathbbm{1}_E (T^{a_n}x) > 1-\varepsilon \Big\}
                > 
                1 - 2\varepsilon.
            \end{equation*}
        \end{lemma}
        
        This result follows from a standard application of the Rokhlin lemma (see \cite[Proposition 2.1]{Jones_2004} for more details).

\section{Asymptotic Behavior of \( \Omega(n) \)}
\label{chp: Number Theory}
    \subsection{Classical Results}
    The study of the asymptotic behavior of \( \Omega(n) \) has a rich history in multiplicative number theory, and such questions are often related to fundamental questions about the prime numbers. For instance, let \( \lambda(n) := (-1)^{\Omega(n)} \) denote the Liouville function. Then the Prime Number Theorem is equivalent to the assertion that
    \[
        \limAVG \lambda(n) = 0.
    \]
    In other words, asymptotically, \( \Omega(n) \) is even half the time. Further clasical results include the Pillai-Selberg Theorem \cite{Pillai1940, Selberg1939} and Erd\H{o}s-Delange Theorem \cite{Erdos1946, Delange1958}. Hardy and Ramanujan \cite[Theorem C']{HR1917} showed that the normal order of \( \Omega(n) \) is \( \log \log n \). 

    \begin{theorem}[(Hardy-Ramanujan Theorem)]
    \label{thm: Hardy Ramanujan}
        For \(C > 0\), define \(g_C: \N \to \N\) by 
        \begin{equation}
        \label{eqn: HR}
            g_C(N) 
            = 
            \# \Big\{ n \leq N : |\Omega(n) - \log \log n| > C \sqrt{\log \log N} \Big\}.
        \end{equation}
        Then for all \(  \varepsilon > 0 \), there is some \( C \geq 1 \) such that 
        \[
            \limsup_{N \to \infty} \frac{g_C(N)}{N} \leq \varepsilon.
        \]
    \end{theorem}

    \begin{remark}
        By \cite[Theorem C]{HR1917}, it is equivalent to verify Equation \eqref{eqn: HR} with \( \log \log n\) replaced by \( \log \log N\). 
    \end{remark}

    Erd\H{o}s and Kac \cite{EK1940} later generalized \cref{thm: Hardy Ramanujan}, proving that \(\Omega(n)\) becomes normally distributed within such intervals.

    \begin{theorem}[(Erd\H{o}s-Kac Theorem)]
    \label{thm: Erdos-Kac}
        For \( A < B \in \R \), 
        \[
            \lim_{N \to \infty} \frac{1}{N} \# \Big\{ n \leq N \,:\, A \leq \frac{\Omega(n) - \log\log N}{\sqrt{\log\log N}} \leq B \Big\}
            = 
            \frac{1}{\sqrt{2\pi}} \int_{A}^B e^{-t^2/2} d t.
        \]
    \end{theorem} 
    
    Thus, the Erd\H{o}s-Kac Theorem states that for large \(N\), the sequence \( \{\Omega(n)\}_{n \leq N} \) becomes roughly normally distributed with mean and variance \(\log \log N\). A result of R\'enyi and Tur\'an \cite{RenyiTuran58} gives the optimal rate of convergence in the Erd\H{o}s-Kac Theorem. 

    \begin{lemma}[\cite{RenyiTuran58}]
    \label{lem:RT} The estimate 
        \[
            \frac{1}{N} \# \Big\{ n \leq N \,:\, \frac{|\Omega(n) - \log \log N|}{\sqrt{\log \log N}} < C \Big\} 
            = 
            \frac{1}{\sqrt{2 \pi}} \int_{-C}^C e^{-t^2/2} \, dt + O\Big( \frac{1}{\sqrt{\log \log N}} \Big).
        \]
        holds uniformly for \( C \in \R \). 
    \end{lemma}

\subsection{Distribution of k-almost primes}
\label{ss: NT Weights}
    In this paper, we convert weighted ergodic averages to those involving weight functions encoding information about the distribution of \(k\)-almost primes. Given \( k \in \N\), let \( \P_k\) denote the set of \(k\)-almost primes, those integers satisfying \( \Omega(n) = k \). Let \( \pi_N(k) \) denote the number of \(k\)-almost primes less than \( N \).  Then 
    \[
        \pi_N(k) = \# \{ n \leq N \,:\, \Omega(n) = k\} = \sum_{n \in \P_k \cap [N]} 1.
    \]
    Erd\H{o}s \cite[Theorem II]{Erdos1948} provides a uniform estimate for \( \pi_N(k)\) for certain values of \( k \). For \( N \in \N \), define the interval \( I_N \) by 
    \[
        I_N = \Big[ \log \log N - C \sqrt{\log \log N} ~,~ \log \log N + C \sqrt{\log \log N} 
        \Big],
    \]
    where \(C \geq 1\) is the constant guaranteed by \cref{thm: Hardy Ramanujan}.
    
    \begin{lemma}[\cite{Erdos1948}] 
    \label{lem: Erdos piN Approx}
        The function \( \pi_N(k) \) satisfies  
        \[
            \pi_N(k) = \frac{N}{\log N} \frac{(\log\log N)^{k-1}}{(k-1)!} (1 + o_{N \to \infty}(1))
        \]
        uniformly for \( k \in I_N \). 
    \end{lemma}
    Erd\H{o}s \cite{Erdos1948} proves a similar estimate for the logarithmically weighted sums 
    \[
        \xi_N(k) := \sum_{n \in \P_k \cap [N]} \frac{1}{n}.
    \]
    Note we can view \(\xi_N\) as logarithmic sums of the arithmetic function \( \mathbbm{1}_{\P_k}\). 

    \begin{lemma}[\cite{Erdos1948}]
    \label{lem: Erdos Logarithmic Approx}
        The function \( \xi_N(k) \) satisfies 
        \[
            \xi_N(k) = \frac{1}{k!}(\log \log N)^k (1+o_{N\to \infty}(1))
        \]
        uniformly for \( k \in I_N \).  
    \end{lemma}

    Thus, following appropriate normalization, the set of \( k\)-almost primes exhibits the same distribution on both the Ces\`aro and logarithmic scales. In the case of double logarithmic weights, the picture changes. Define 
    \[
        \eta_N(k) := \sum_{n \in \P_k \cap [N]} \frac{1}{n \log n}.
    \]  
    
    The study of the values of \( \eta_N(k) \) has a long history. First, Erd\H{o}s \cite{Erdos_1935} proved that the \( \eta_N(k) \) are uniformly bounded. In fact, the result holds for the infinite sums
    \[
        \eta(k) := \sum_{n \in \P_k } \frac{1}{n \log n},
    \]
    also known as the \textit{Erd\H{o}s sums of k-almost primes}. 

    \begin{lemma}[\cite{Erdos_1935}] 
    \label{lem: Erdos Eta}
        There exist a constant \( C > 0 \) such that \( \eta(k) \leq C \) for all \( k \in \N \).
    \end{lemma}

    It was later shown by Zhang \cite{Zhang_1993} that this constant \( C \) is given by \( \eta(1)\). In other words, the primes have the maximal Erd\H{o}s sum. The strongest known approximation for \( \eta(k) \) is due to Gorodetsky, Lichtman, and Wong \cite{GLW2023}, who demonstrate that for large \(k\), the function \(\eta(k)\) is monotonically increasing to 1. 
    
    \begin{prop}[\cite{GLW2023}]
    \label{lem: GLM Estimate}
        There exists a constant \(d>0\) such that for any \(k\in \N\)
        \begin{equation*}
            \eta(k)
            = 
            1- \frac{\log2}{4} 2^{-k} \Big( d k^2+ O(k \log (k+1)\Big)
        \end{equation*}
        uniformly for \( k \geq 1\). 
    \end{prop}

    \begin{remark}
        Since \( \eta_N(k) \leq \eta(k) \) for all \( N \in \N\), \cref{lem: Erdos Eta} holds for \( \eta_N(k) \) as well. Similarly, \cref{lem: GLM Estimate} implies \( \eta_N(k) < 1 \) for large \( k \). These properties are essential to the proof of \cref{thm: Double Log Convergence}.
    \end{remark} 

\section{Proof of \cref{thm: Double Log Convergence}}
\label{sec: Proof of Double Log Convergence}

    In the following, let \( L_N = 2 \log \log N \). We treat the terms of the ergodic averages up to and following \( L_{N} \) separately:
        \begin{equation}
        \label{eqn: double log split}
            \underset{n \in [N]}{\A^{\log \log}} f(T^{\Omega(n)}x) 
            = 
            \underset{k \in [N]}{\A^{\eta_{N}}} f(T^k x)
            = 
            \underbrace{\underset{k \in [L_{N}]}{\A^{\eta_{N}}} f(T^k x) }_{(1)} + \underbrace{\underset{k \in [L_{N},N]}{\A^{\eta_{N}}} f(T^k x)}_{(2)}
        \end{equation}

    The proof of \cref{thm: Double Log Convergence} is composed of three distinct parts. We first demonstrate convergence for \( L^\infty\) functions, a dense set of functions in \( L^1\). 

    \begin{prop}
    \label{prop: Double Log for Bdd}
        Let \( (X, \mu, T ) \) be an ergodic system. Then for any \( f \in L^\infty(\mu)\), 
        \[
            \lim_{N \to \infty} \Ellog f(T^{\Omega(n)}x) = \int f \,d\mu 
        \]
        for \(\mu \)-almost every \( x \in X\). 
    \end{prop}

    The proof of \cref{prop: Double Log for Bdd} is delayed until \cref{subsec: Conv for Bdd}. We then demonstrate the existence of a suitable maximal inequality over the integers, whose proof can also be found in \cref{subsec: Conv for Bdd}.
    
    \begin{prop}
    \label{prop: MEI for loglog}
        There exists a \( D > 0 \) such that for any finitely supported non-negative function \( \varphi: \N \to \C \), 
        \[
            \# \Big\{ j \in \N \,:\, \sup_{N\in \N} \frac{1}{\log\log N} \sum_{k \leq L_N} \eta_N(k) \varphi(j+k)>1 \Big\} \leq D  \norm{\varphi}_{\ell^1}.
        \]
    \end{prop}
    
    The Calder\'on transference principle (see for instance \cite{Calderon1968}) then implies \cref{prop: MEI for MPS}, the desired maximal inequality in an arbitrary measure-preserving system. A standard approximation argument, given below, then implies the desired convergence of Term (1) in Equation \eqref{eqn: double log split}.
    
    In \cref{subsec: Conv for L1}, we treat the more delicate case of pointwise convergence of the tail, Term (2) in Equation \eqref{eqn: double log split}. To accomplish this, we first reduce the pointwise averages to those along certain lacunary sequences. 
    
    \begin{lemma}\label{lem:reduction}
        For \( \rho \in \R \) and \( i \in \N \), set \( N_i
        =
        \lfloor2^{2^{\rho^i}}\rfloor.
        \) 
        Let \( f\geq 0 \) and suppose that there exists \( L \in \R \) such that
        \[ 
            \lim_{i \to \infty }\underset{k \in [N_i]}{\A^{\eta_{N_i}}} f(T^k x) 
            = 
            L
        \]
        for every \( \rho > 1 \). Then 
        \[ 
            \lim_{N \to \infty} \underset{k \in [N]}{\A^{\eta_N}} f(T^k x) 
            = 
            L.
        \] 
    \end{lemma}

    We then show that for any \( \rho > 1\), Term \((2) \) converges to zero pointwise. 
    
    \begin{prop}
    \label{prop: Tail Term}
        Let \( (X, \mu, T ) \) be an ergodic system. For \( \rho > 1\), define \( N_i := 2^{2^{\rho^i}} \). Then for any \( f \in L^1(\mu)\), 
        \[
            \lim_{i \to \infty} \underset{k \in [L_{N_i},N_i]}{\A^{\eta_{N_i}}} f(T^k x) = 0
        \]
        for \( \mu-\)almost every \( x \in X\).
    \end{prop}
    
    We now present the proof of \cref{thm: Double Log Convergence}, assuming Propositions \ref{prop: MEI for MPS} and \ref{prop: Double Log for Bdd} and \cref{lem:reduction}.  

    \begin{proof}[Proof of \cref{thm: Double Log Convergence}]
        Let \( (X, \mu, T) \) be an ergodic system and take \( f \in L^1(\mu) \). For \( N \in \N\),
        \begin{equation}
        \label{eqn: double log 1}
            \underset{n \in [N]}{\A^{\log \log}} f(T^{\Omega(n)}x) 
            %= 
            %\underset{k \in [N]}{\A^{\eta_{N}}} f(T^k x)
            = 
            \underset{k \in [L_{N}]}{\A^{\eta_{N}}} f(T^k x)  + \underset{k \in [L_{N},N]}{\A^{\eta_{N}}} f(T^k x).
        \end{equation}
        Let \( \varepsilon > 0 \) and take \( g \in L^\infty(\mu) \) such that \( \norm{f-g}_1 < \varepsilon^2 \). Then 
        \[
            \norm{\underset{k \in [L_{N}]}{\A^{\eta_{N}}} f(T^k x) - \underset{k \in [L_{N}]}{\A^{\eta_{N}}} g(T^k x)}_1 \leq \norm{f-g}_1,
        \]
        so that \( \big|\int f \, d\mu - \int g \, d\mu \big| < \varepsilon^2 \). Observe that
        \begin{align*}
            \limsup_{N \to \infty} \Big| &\underset{k \in [L_{N}]}{\A^{\eta_{N}}} f(T^k x) - \int f \, d\mu \Big| 
            \\
            &\leq 
            \sup_{n \geq 1} \Big| \underset{k \in [L_{N}]}{\A^{\eta_{N}}} (f-g)(T^k x) \Big| + \limsup_{N \to \infty} \Big| \underset{k \in [L_{N}]}{\A^{\eta_{N}}} g(T^k x) - \int g \,d\mu \Big| + \Big|\int f \, d\mu - \int g \,d\mu \Big| 
        \end{align*}
        By \cref{prop: Double Log for Bdd},
        \[
            \limsup_{N \to \infty} \Big| \underset{k \in [L_{N}]}{\A^{\eta_{N}}} g(T^k x) - \int g \,d\mu \Big| = 0.
        \]
        Hence
        \begin{align*}
            \mu \Big\{ x \in X \,:\, \limsup_{N \to \infty} &\Big| \underset{k \in [L_{N}]}{\A^{\eta_{N}}} f(T^k x) \Big| > 2\varepsilon \Big\}
            \\
            & \leq 
            \mu \Big\{ x \in X \,:\, \sup_{n \geq 1} \Big| \underset{k \in [L_{N}]}{\A^{\eta_{N}}} (f-g)(T^k x) \Big| + \Big|\int f \, d\mu - \int g \,d\mu \Big| > 2\varepsilon \Big\}
            \\
            & \leq
            \mu \Big\{ x \in X \,:\, \sup_{n \geq 1} \Big| \underset{k \in [L_{N}]}{\A^{\eta_{N}}} (f-g)(T^k x) \Big| > \varepsilon \Big\}
            \\
            & \leq 
            \frac{C}{\varepsilon} \norm{f-g}_1
            \\
            & \leq
            C\varepsilon,
        \end{align*}
        where the second to last inequality follows from \cref{prop: MEI for MPS}. Hence
        \begin{equation}
        \label{eqn: double log 2}
            \lim_{N \to \infty} \underset{k \in [L_{N}]}{\A^{\eta_{N}}} f(T^k x) = \int f \, d\mu
        \end{equation}
        for almost every \( x \in X\). 
        
        Now, let \( \rho > 1\). By \cref{lem:reduction}, it is enough to consider Equation \eqref{eqn: double log 1} along the sequence \( N_i = \floor{2^{2^{\rho^i}}}\). Since Equation \eqref{eqn: double log 2} holds along the integers, it also holds along the sequence \( N_i \). Then, by Equation \eqref{eqn: double log 2} and \cref{prop: Tail Term}, 
        \[
            \lim_{i \to \infty} \underset{n \in [N_i]}{\A^{\log \log}} f(T^{\Omega(n)}x) 
            =
            \lim_{i \to \infty} \underset{k \in [L_{N_i}]}{\A^{\eta_{N_i}}} f(T^k x)  + \lim_{i \to \infty} \underset{k \in [L_{N_i},N_i]}{\A^{\eta_{N_i}}} f(T^k x)
            = 
            \int f \, d\mu
        \]
        for \(\mu-\)almost every \( x \in X\), completing the proof.
    \end{proof}

    \subsection{Proofs of \cref{prop: Double Log for Bdd} and \cref{prop: MEI for loglog}}
    \label{subsec: Conv for Bdd}
        
        To prove \cref{prop: Double Log for Bdd}, we require the following lemma, guaranteeing that, for large \( N\), the weights \( \eta_N(k) \) are largely supported on an initial segment of \( [1, N]\). For \( N \in \N \), let \( L_N = \log \log N + C \sqrt{\log \log N} \), where \( C \) is the constant guaranteed by \cref{thm: Hardy Ramanujan}.
         
        \begin{lemma}
        \label{lem: Eta weight distribution}
            Let \( \varepsilon > 0\). The function \( \eta_N(k) \) satisfies
            \begin{equation*}
                \sum_{k \leq L_N} \eta_N(k)
                \geq 
                (1-\varepsilon) \log \log N.
            \end{equation*}
            for sufficiently large \( N\).
        \end{lemma}

        \begin{remark}
            Notably, there are non-negligible contributions to the support of \( \eta_N(k) \) in the interval \( [1, \log \log N - C \sqrt{\log \log N}]\). Thus, \cref{lem: Eta weight distribution} cannot be strengthened to \( k \) only within the interval \( I_N \), as in Equation \eqref{eqn: Weight Support}. In fact, if it could, \( \Omega(n) \) would be inherently strong sweeping out.
        \end{remark}

        \begin{proof}
            Let \( \varepsilon > 0 \) and let \( C > 0 \) be that given by the Hardy-Ramanujan Theorem. For \( N \in \N \), define \(L_N := \log \log N + C \sqrt{\log \log N} \) and \( J_N := [L_N \,,\, \infty) \). Then
            \begin{align}
            \label{eqn: etaa 1}
                \sum_{k \geq L_N} \eta_N(k) 
                &= 
                \sum_{n \leq N} \frac{1}{n \log n} \mathbbm{1}_{J_N}(\Omega(n)) 
                \\
                \notag &= \frac{1}{N \log N} \sum_{n \leq N} \mathbbm{1}_{J_N}(\Omega(n)) \\ 
                \notag & \quad + \sum_{m \leq N-1} \Bigg( \frac{1}{m \log m} - \frac{1}{(m+1) \log (m+1) }\Bigg) \sum_{n \leq m} \mathbbm{1}_{J_N}(\Omega(n))
            \end{align}
            By the Hardy-Ramanujan Theorem, there exists \( N_0 \in \N \) such that for \( m \geq N_0\), 
            \[
                \frac{1}{m} \sum_{n \leq m } \mathbbm{1}_{(I_m)^c} (\Omega(n)) 
                = 
                \frac{1}{m} \sum_{k \in (I_m)^c} \pi_m(k) 
                \leq 
                \varepsilon,  
            \]
            where \( I_m = [ \log \log m - C\sqrt{\log \log m} \,,\,\log \log m + C\sqrt{\log \log m}]\). Take \( N > N_0 \). Notice that \( J_N \subset (I_m)^c \) for all \( m \leq N \). Then
            \begin{align}
            \label{eqn: etaa 2}
                \notag \frac{1}{\log N} \Bigg( \frac{1}{N}\sum_{n \leq N} \mathbbm{1}_{J_N}(&\Omega(n)) \Bigg) 
                + \sum_{m \leq N-1} \Bigg( \frac{1}{m \log m} - \frac{1}{(m+1) \log (m+1) }\Bigg) \Bigg( \sum_{n \leq m} \mathbbm{1}_{J_N}(\Omega(n)) \Bigg) 
                \\
                \notag &\leq 
                \varepsilon \Bigg( \frac{1}{\log N} + \sum_{m = N_0}^{N-1} m \Bigg( \frac{1}{m \log m} - \frac{1}{(m+1) \log (m+1) }\Bigg) \Bigg) + C_\varepsilon
                \\
                &\leq \varepsilon \sum_{m \leq N} \frac{1}{m \log m} + C_\varepsilon. 
            \end{align} 
            Now, take \( N \) large enough so that \( C_\varepsilon \leq \varepsilon \log \log N \). Then Equations \eqref{eqn: etaa 1} and \eqref{eqn: etaa 2} imply
            \[
                \sum_{k \geq L_N} \eta_N(k) \leq 2 \varepsilon \log \log N,
            \]
            and we are done. 
        \end{proof}

        \begin{proof}[Proof of \cref{prop: Double Log for Bdd}]
            Let \( (X, \mu , T )\) be an ergodic system and \( f \in L^\infty(\mu)\). Set \( M = \|f\|_\infty \). Regrouping by value of \( \Omega(n) \),
                \[
                    \Ellog f(T^{\Omega(n)}x) = \underset{k \in [N]}{\A^{\eta_N}} f(T^k x).
                \]
                Let \( \varepsilon > 0\). Define \( L_N := \log \log N + C \sqrt{\log \log N}\). By \cref{lem: Eta weight distribution}, 
                \[
                    \Big| \underset{k \in [N]}{\A^{\eta_N}} f(T^k x) - \underset{k \in [L_N]}{\A^{\eta_N}} f(T^k x) \Big| 
                    = 
                    \frac{M}{\log \log N} \sum_{k = L_N}^N \eta_N(k) \leq M \varepsilon. 
                \]
                Then 
                \begin{equation}
                    \underset{k \in [N]}{\A^{\eta_N}} f(T^k x) = \underset{k \in [L_N]}{\A^{\eta_N}} f(T^k x) + O(\varepsilon).  
                \end{equation}
                We now proceed by comparison of the (slightly altered) averages \( \underset{k \in [L_N]}{\A} \eta_N(k) f(T^k x) \) to the standard Birkhoff averages. The justification for this alteration is provided at the end of the proof. 
                
                By \Cref{lem: GLM Estimate}, there exists \( K_0 > 0\) such that \( \eta_N(k) < 1 \) for all \( k \geq K_0 \). Then 
                \begin{align}
                \label{eqn: LOGLOG1}
                    \notag \Big| \underset{k \in [L_N]}{\A} f(T^k x) - \underset{k \in [L_N]}{\A} \eta_N(k) f(T^k x) \Big| 
                    &= 
                    \frac{1}{L_N} \Big|\sum_{k \leq L_N }\big(1- \eta_N(k)\big) f(T^ k x)\Big|
                    \\
                    \notag &\leq 
                    \frac{M}{L_N} \sum_{k \leq L_N }|1- \eta_N(k)|
                    \\
                    & = \frac{M}{L_N} \sum_{k \leq K_0-1} |1 - \eta_N(k) | + \frac{M}{L_N} \sum_{k = K_0}^{L_N} (1 - \eta_N(k)).
                \end{align}
                By \cref{lem: Erdos Eta}, \( \eta_N(k) \leq \eta(1) \) for all \(k, N \in \N\), so that 
                \begin{equation*}
                    \sum_{k \leq K_0 -1 }|1-\eta_N(k)|
                    \leq 
                    K_0 - 1 + \sum_{k \leq K_0 -1} \eta(1)
                    = (K_0-1)(1+\eta(1)).
                \end{equation*}
                Hence
                \begin{equation}
                \label{eqn: LOGLOG2}
                    \lim_{N \to \infty} \frac{M}{L_N} \sum_{k \leq K_0-1} |1 - \eta_N(k) | = 0.
                \end{equation}
                Additionally, by \cref{lem: Eta weight distribution},  
                \begin{align*}
                    \sum_{k = K_0}^{L_N} (1-\eta_N(k))
                    &= 
                    L_N - K_0 - \sum_{k \leq L_N}\eta_N(k)+ \sum_{k \leq K_0 }\eta_N(k)
                    \\
                   &\leq 
                   L_N - K_0 - (1- \varepsilon)  \log \log N + D
                    \\
                    &\leq 
                    \varepsilon \log \log N + C \sqrt{ \log \log N}- K_0 + D, 
                \end{align*}
                where \( D:= \sum_{k \leq K_0} \eta(k) \). Then, 
                \begin{equation}
                \label{eqn: LOGLOG3}
                    \frac{M}{L_N}\sum_{k = K_0}^{L_N} (1-\eta_N(k)) = O(\varepsilon).  
                \end{equation}
                Combining Equations \eqref{eqn: LOGLOG1}, \eqref{eqn: LOGLOG2}, and \eqref{eqn: LOGLOG3}, 
                \begin{align}
                \label{eqn: LOGLOG4}
                    \underset{k \in [L_N]}{\A} \eta_N(k) f(T^k x) 
                    = 
                    \underset{k \in [L_N]}{\A} f(T^k x) + O(\varepsilon)
                    =
                    \int f \, d\mu
                    +
                    O(\varepsilon), 
                \end{align}
                as desired. Our alteration of the normalization from \( 1/\log\log N \) to \( 1/L_N\) does not affect Equation \eqref{eqn: LOGLOG4} since  
                \begin{align*}
                    \underset{k \in [L_N]}{\A^{\eta_N}} f(T^k x) - \underset{k \in [L_N]}{\A} \eta_N(k) f(T^k x)
                    &=
                    \Big( \frac{1}{\log \log N } - \frac{1}{L_N} \Big) \sum_{k \leq L_N} \eta_N(k) f(T^k x) 
                    \\
                    &=
                    \frac{1}{\sqrt{\log\log N}} \underset{k \in [L_N]}{\A} \eta_N(k) f(T^k x), 
                \end{align*} 
                so that
                \[
                    \underset{k \in [L_N]}{\A^{\eta_N}} f(T^k x) 
                    = 
                    \underset{k \in [L_N]}{\A} \eta_N(k) f(T^k x) \Big (1 + O\Big(\frac{1}{\sqrt{\log\log N}}\Big)\Big).
                \]
        \end{proof}

        \begin{proof}[Proof of \cref{prop: MEI for loglog}]  
            Let  \( \varphi \geq 0 \) be a  finitely supported function on the integers. Define \( L_N := 2 \log \log N \) and set
            \[
                E := \Big\{ j \in \N \,:\, \sup_{N\in\N} \frac{1}{\log\log N} \sum_{k \leq L_N} \eta_N(k) \varphi(j+k)>1 \Big\}.
            \]
            We want to find a constant \( D > 0\), not depending on \( \varphi \), such that \( \#E \leq D \norm{\varphi}_{\ell^1} \). We proceed inductively. Set \( E_1 = E\). Take \( j_1 \) to be the minimal element of \( E_1\). By definition of \( E \), there is an integer \( N_1 \) such that 
            \begin{equation}
            \label{eqn: MEI Eta1}
                \sum_{k \leq L_{N_1}} \eta_{N_1}(k) \varphi(j_1+k)
                >
                \log \log N_1.
            \end{equation}
            Set \( I_1 = [j_1, j_1 + L_{N_1}] \). Then
            \begin{equation}
            \label{eqn: MEI Eta2}
                \# (E \cap I_1) \leq L_{N_1} = 2 \log \log N_1.   
            \end{equation}
            Additionally, by \cref{lem: Erdos Eta}, 
            \begin{equation}
            \label{eqn: MEI Eta3}
                \sum_{k \leq L_{N_1}} \eta_{N_1}(k) \varphi(j_1+k)
                \leq 
                \sum_{k \leq L_{N_1}} \eta(k) \varphi(j_1+k)
                \leq 
                \eta(1) \sum_{k \leq L_{N_1}} \varphi(j_1+k).
            \end{equation}
            Combining, Equations \eqref{eqn: MEI Eta1}, \eqref{eqn: MEI Eta2}, and \eqref{eqn: MEI Eta3},  
            \[
                \# (E \cap I_1) \leq 2 \,\eta(1) \sum_{j \in I_1} \varphi(j). 
            \]
            Now, set \( E_2 = E_1 \setminus I_1 \). Take \( j_2 \) to be the minimal element of \( E_2\) and take \( N_2 \) such that 
            \begin{equation}
            \label{eqn: MEI Eta4}
                \sum_{k \leq L_{N_2}} \eta_{N_2}(k) \varphi(j_2+k)
                >
                \log\log N_2.
            \end{equation}
            Set \( I_2 = [j_2, j_2 + L_{N_2}] \). Notice that, by construction, \( I_1 \) and \( I_2 \) are disjoint intervals. An identical argument as before shows that 
            \[
                \# (E \cap I_2) \leq 2 \,\eta(1) \sum_{j \in I_2} \varphi(j). 
            \] 
            Iterating, this process terminates after a finite number of steps, say \( K\), since \( \varphi \) is finitely supported. This yields \( K \) intervals \( I_k \), each satisfying 
            \[
                \# (E \cap I_k) \leq 2 \,\eta(1) \sum_{j \in I_k} \varphi(j). 
            \]
            By construction, these intervals are pairwise disjoint. Additionally, since the integer \( j_k \) was chosen to be minimal at each step, \( E \sse \cup_{k \leq K} I_k \). Hence 
            \[
                \# E = \sum_{k\leq K} \# ( E \cap I_k) \leq 2 \,\eta(1) \sum_{k \leq K} \sum_{j \in I_k} \varphi(j) \leq 2 \,\eta(1) \norm{\varphi}_{\ell^1},
            \]
            Taking \( D = 2 \, \eta(1) \), we are done. 
        \end{proof}

    \subsection{Proofs of \cref{lem:reduction} and \cref{prop: Tail Term}}
    \label{subsec: Conv for L1}

        \begin{proof}[Proof of \cref{lem:reduction}]
        Fix \( \rho > 1 \) and define \( N_i = N_i(\rho) := \lfloor 2^{2^{\rho^i}}\rfloor \). Let \( N \) be such that \( N_i \leq N < N_{i+1} \). Observe that 
        \begin{align*}
            \frac{1}{\rho} \cdot \underset{k \in [N_{i}]}{\A^{\eta_{N_{i}}}} f(T^k x) 
            \leq 
            \underset{k \in [N]}{\A^{\eta_N}} f(T^k x) 
            \leq 
            \rho \cdot \underset{k \in [N_{i+1}]}{\A^{\eta_{N_{i+1}}}} f(T^k x). 
        \end{align*}
        This implies
        \begin{align*}
            \frac{1}{\rho} \cdot \lim_{i \to \infty} \underset{k \in [N_{i}]}{\A^{\eta_{N_{i}}}} f(T^k x) 
            \leq 
            \liminf_{N \to \infty} \underset{k \in [N]}{\A^{\eta_N}} f(T^k x) 
            \leq 
            \limsup_{N \to \infty} \underset{k \in [N]}{\A^{\eta_N}} f(T^k x) 
            \leq 
            \rho \cdot \lim_{i \to \infty} \underset{k \in [N_{i+1}]}{\A^{\eta_{N_{i+1}}}} f(T^k x). 
        \end{align*}
        By assumption, \( \lim_{i \to \infty} \underset{k \in [N_{i}]}{\A^{\eta_{N_{i}}}} f(T^k x) = L \). Hence, we have
        \begin{equation*}
            \frac{L}{\rho} \leq \liminf_{N}\underset{k \in [N]}{\A^{\eta_N}} f(T^k x) 
            \leq 
            \limsup_{N}\underset{k \in [N]}{\A^{\eta_N}} f(T^k x) 
            \leq 
            \rho L.
        \end{equation*}
        Letting $\rho \to 1$, we are done.
    \end{proof}

        We now prove \cref{prop: Tail Term}. By application of the Markov Inequality, it is not hard to show that the desired averages converge to zero in the \( L^1\)-norm. To demonstrate pointwise convergence in \( L^1 \), we require information regarding the rate of convergence in \cref{lem: Eta weight distribution}.

        \begin{lemma}
        \label{lem: eta rate}
            Define \( L_N := 2 \log \log N \). The function \( \eta_N(k) \) satisfies
            \[
                \frac{1}{\log \log N} \sum_{L_N \leq k \leq N} \eta_N(k) = O \Big(\frac{1}{\sqrt[\leftroot{-2}\uproot{2}4]{\log \log N}} \Big). 
            \]
        \end{lemma}

        \begin{proof}
            For \( N \in \N \), define \(L_N := 2 \log \log N \) and \( J_N := [L_N \,,\, \infty) \). Then
            \begin{align}
            \label{eqn: eta1}
                \sum_{L_N \leq k \leq N} \eta_N(k) 
                &= 
                \sum_{n \leq N} \frac{1}{n \log n} \mathbbm{1}_{J_N}(\Omega(n)) 
                \\
                \notag &= \frac{1}{N \log N} \sum_{n \leq N} \mathbbm{1}_{J_N}(\Omega(n)) \\ 
                \notag & \quad \quad + \sum_{m \leq N-1} \bigg( \frac{1}{m \log m} - \frac{1}{(m+1) \log (m+1) }\bigg) \sum_{n \leq m} \mathbbm{1}_{J_N}(\Omega(n))
            \end{align}
            For convenience, define 
            \[ 
                w(m) := m \bigg(\frac{1}{m \log m} - \frac{1}{(m+1) \log (m+1) } \bigg).
            \]
            Then 
            Equation \eqref{eqn: eta1} yields
            \begin{equation}
            \label{eqn: eta2}
                \sum_{L_N \leq k \leq N} \eta_N(k)
                \leq 
                \frac{1}{\log N} \EC{} \mathbbm{1}_{J_N}(\Omega(n)) + \sum_{m \leq N-1} w(m) \ECC{n}{[m]} \mathbbm{1}_{J_N}(\Omega(n))
            \end{equation}
            For \( m \in \N \), define \( C_m := \sqrt{\log \log m}\) and 
            \[
                I_m := \big[ \log \log m - C_m \sqrt{\log \log m} \,,\, \log \log m + C_m \sqrt{\log \log m} \big] = \big[ 0 \,,\, 2 \log \log m \big].
            \]
            By \cref{lem:RT}, there exists \( N_0, D \in \N \) such that for \( m \geq N_0\), 
            \[
                \ECC{n}{[m]} \mathbbm{1}_{(I_m)^c} (\Omega(n)) 
                = 
                \frac{1}{m} \# \Big\{ \frac{n \leq m \,:\, |\Omega(n) - \log \log m|}{\sqrt{\log \log m}} < C_m \Big\} 
                \leq 
                \frac{D}{\sqrt{\log \log m}}.  
            \]
            Now, define \( N_1 = N_1(N) \) by 
            \[
                \log \log N_1 := \sqrt{\log \log N}.
            \]
            Take \( N \) large enough so that \( N > N_1 > N_0 \). Notice that \( J_N \subset (I_m)^c \) for all \( m \leq N \). Then
            \begin{align*}
                \frac{1}{\log N} \EC{} &\mathbbm{1}_{J_N}(\Omega(n)) + \sum_{m \leq N-1} w(m) \ECC{n}{[m]} \mathbbm{1}_{J_N}(\Omega(n))
                \\
                &\leq
                \frac{1}{\log N} \EC{} \mathbbm{1}_{(I_m)^c}(\Omega(n)) + \sum_{m \leq N-1} w(m) \ECC{n}{[m]} \mathbbm{1}_{(I_m)^c}(\Omega(n))
                \\
                & \leq
                \frac{1}{\log N} \Big( \frac{D}{\sqrt{\log \log N}} \Big) + \sum_{m \leq N_1} w(m) 
                + \sum_{N_1 \leq m \leq N-1} w(m) \Big( \frac{D}{\sqrt{\log \log m}} \Big)   
                \\
                &\leq 
                \frac{D}{\sqrt{\log \log N_1}} \bigg( \frac{1}{\log N} + \sum_{m = N_1}^{N-1} w(m) \bigg) + \sum_{n \leq N_1} w(m)
            \end{align*} 
            Now, notice that for any \( M \in \N \),
            \[
                \frac{1}{\log M} + \sum_{n \leq M} w(m) = \sum_{n \leq M} \frac{1}{n \log n} = \log \log M + o(1/\log\log M),
            \]
            so that 
            \begin{align}
            \label{eqn: eta3}
                \frac{D}{\sqrt{\log \log N_1}} \bigg( \frac{1}{\log N} &+ \sum_{m = N_1}^{N-1} w(m) \bigg) + \sum_{n \leq N_1} w(m)
                \\
                \notag &\leq
                \frac{D}{\sqrt{\log \log N_1}} \log \log N + \log \log N_1
                \\
                \notag &= \Big( \frac{D}{\sqrt[\leftroot{-2}\uproot{2}4]{\log \log N}} + \frac{1}{\sqrt{\log \log N} }\Big) \log \log N,
            \end{align}
            where the final equality follows by definition of \( N_1\). Combining Equations \eqref{eqn: eta2} and \eqref{eqn: eta3},
            \[
                \frac{1}{\log \log N}\sum_{k \geq L_N} \eta_N(k) 
                \leq 
                \frac{D+1}{\sqrt[\leftroot{-2}\uproot{2}4]{\log \log N}},
            \]
            and we are done. 
        \end{proof}

        \begin{proof}[Proof of \cref{prop: Tail Term}]
            Let \( \rho > 1 \) and define \( N_i = \lfloor 2^{2^{\rho^i}} \rfloor \).
            By \cref{lem: eta rate}, there exists \( D, N' \in \N \) such that for \( N \geq N'\), 
            \begin{equation}
            \label{eqn: RT}
                \frac{1}{\log \log N}\sum_{L_N \leq k \leq N}\eta_N(k) \leq \frac{D}{\sqrt[\leftroot{-2}\uproot{2}4]{\log \log N}}. 
            \end{equation}

            Without loss of generality, take \( f \in L^1(\mu) \) non-negative. For each \( i \in \N \), define \( \lambda_i = 1/i \) and 
            \[
                E_i := \Big\{ x \in X \,:\, \underset{k \in [L_{N_i},N_i]}{\A^{\eta_{N_i}}} f(T^k x) > \lambda_i \Big\}.
            \]
            By Markov's Inequality and Equation \eqref{eqn: RT}, 
            \begin{align}
            \label{eqn: LOGLOG5} 
                \mu (E_i)
                &\leq 
                \frac{1}{\lambda_i} \int \frac{1}{\log\log N_i} \sum_{L_{N_i} \leq k \leq N_i} \eta_{N_i}(k) f(T^k x) \, d\mu
                \leq 
                \frac{D \norm{f}_1}{\lambda_i \sqrt[\leftroot{-2}\uproot{2}4]{\log \log N_i}}
            \end{align}
            for each \( N_i \geq N'\). By definition,
            \(
                \lambda_i \sqrt[\leftroot{-2}\uproot{2}4]{\log \log N_i} = \rho^{i/4}/i. 
            \)
            Since \( \sum_{i \geq 1} i/\rho^{i/4} < \infty \) for any \( \rho > 1 \), the Borel-Cantelli Lemma implies that for almost every \( x \in X\), 
            \[
                \underset{k \in [L_{N_i},N_i]}{\A^{\eta_{N_i}}} f(T^k x) < \frac{1}{i}
            \]
            for large enough \( i\). Hence
            \[
               \lim_{i \to \infty} \underset{k \in [L_{N_i},N_i]}{\A^{\eta_{N_i}}} f(T^k x) = 0 ,
            \]
            completing the proof.
        \end{proof}

\section{Proof of \cref{thm: Log Sweeping Out}}
\label{sec: Proof Log SSO}

    In this section, we demonstrate that \( \Omega(n) \) retains the strong sweeping out property for logarithmic averages. We first recount the main ideas in the proof of \cref{thm: Pointwise Nonconvergence} in \cite{Loyd2021}, as the proof of \cref{thm: Log Sweeping Out} follows a similar argument. Regrouping the ergodic averages by the value of \( \Omega(n) \),
    \begin{equation}
    \label{eqn: LOG0}
        \EC{} f(T^{\Omega(n)}x) = \ECC{k}{[N]} \pi_N(k) f(T^k x).
    \end{equation}

    First, \cref{thm: Hardy Ramanujan} implies that
    \[
        \ECC{k}{[N]} \pi_N(k) f(T^k x) 
        = 
        \sum_{k \in I_N} \frac{\pi_N(k)}{N} f(T^k x) + O(\varepsilon). 
    \]

    Combining \cref{lem: Erdos piN Approx} with Stirling's formula, the function \( \pi_N/N \) can be approximated by a Gaussian with mean and variance \( \log \log N \). 

    \begin{lemma}
    \label{lem: Erdos to Gaussian}
    The function \( \pi_N(k) \) satisfies
        \[
            \frac{\pi_{k}(N)}{N} 
            = 
            \frac{1}{\sqrt{2\pi \log \log N}} e^{-\frac{1}{2} \big(\frac{k - \log \log N}{\sqrt{\log \log N}} \big)^2} (1+ o_{N \to \infty}(1))
        \]
    uniformly for \( k \in I_N \).
    \end{lemma}
    
    Then, the strong sweeping out property for \( \Omega(n) \) is a consequence of the following proposition.  

    \begin{prop}[\cite{Loyd2021}]
    \label{prop: Operator SSO}
        The averaging operators 
        \[
            T_N f(x) 
            :=
            \frac{1}{\sqrt{2\pi \log \log N}} \sum_{k \in I_N} e^{- \frac{1}{2}(\frac{k- \log \log N}{\sqrt{\log \log N}})^2} f(T^k x)
        \]
        satisfy the strong sweeping out property, meaning for any non-atomic ergodic system \((X, \mathcal{B}, \mu, T)\) and \( \varepsilon > 0 \), there is a set \( A \in \mathcal{B} \) satisfying \( \mu(A) < \varepsilon\) and 
        \[
            \limsup_{N \to \infty} T_N \mathbbm{1}_A (x) = 1
            \quad \text{ and }
            \liminf_{N \to \infty} T_N \mathbbm{1}_A (x) = 0
        \]
        for \( \mu\)-almost every \( x \in X\). 
    \end{prop}

    Now, consider the function
    \[ 
        \xi _N(k) = \sum_{n \in \P_k \cap [N]} \frac{1}{n}
    \]
    discussed in \cref{ss: NT Weights}. Regrouping the logarithmic averages along \( \Omega(n)\),
    \begin{equation}
    \label{eqn: LogLog Weighted Expression}
        \Elog f(T^{\Omega(n)}x) = \Elogg{k}{[N]} \xi_N(k) f(T^k x) .
    \end{equation}
    
    To prove \cref{thm: Log Sweeping Out}, we first require a logarithmic version of \cref{thm: Hardy Ramanujan}.  
    
    \begin{lemma}
    \label{lem: logarithmic weight support}
        Let \(\varepsilon > 0 \). The function \( \xi_N(k) \) satisfies
        \[
            \frac{1}{\log N} \sum_{k \leq N} \xi_N(k) = \frac{1}{\log N} \sum_{k \in I_N} \xi_N(k) + O(\varepsilon). 
        \]
    \end{lemma}

    \begin{proof}
        Let \( \varepsilon > 0 \). We want to show that 
        \[
            \frac{1}{\log N} \sum_{k \notin I_N} \xi_N(k) = O_{N \to \infty}(\varepsilon).
        \]
        Let \( I^c_N \) denote \( \R \setminus I_N\). Then 
        \begin{align}
        \label{eqn: LOG1}
            \frac{1}{\log N} \sum_{k \notin I_N} \xi_N(k) &= \frac{1}{\log N} \sum_{n \leq N} \frac{1}{n} \, \mathbbm{1}_{I^c_N} (\Omega(n)).     
        \end{align}
        For \( N \in \N \) and each \( n \leq N \), define moving intervals
        \[
            I_{N,n} := [\log\log n - C \sqrt{\log\log N}\,,\, \log\log n + C \sqrt{\log\log N}].
        \]
        Let \( \delta > 0 \). Notice that 
        \[
            \frac{1}{\log N} \sum_{n \leq N^\delta} \frac{1}{n} = \frac{\log (N^\delta)}{\log N} = \delta.
        \]
        For \( N^\delta \leq n \leq N \), \( \log \log n = \log \log N + O(1)\), so that by the Hardy-Ramanujan Theorem, 
        \[
            \frac{1}{n} \# \{ m \leq n \,:\, \Omega(m) \in I_N \} \geq \varepsilon. 
        \]
        Then
        \[
            \frac{1}{\log N} \sum_{N^\delta \leq n \leq N} \frac{1}{n} \big|(\mathbbm{1}_{I^c_N} - \mathbbm{1}_{I^c_{N,n}}) (\Omega(n))\big| 
            \leq 
            \frac{\varepsilon}{\log N} \sum_{N^\delta \leq n \leq N} \frac{1}{n} 
            = 
            \varepsilon(1- \delta)
            < \varepsilon.
        \]
        Hence, 
        \[
            \frac{1}{\log N} \sum_{n \leq N} \frac{1}{n} \, \mathbbm{1}_{I^c_N} (\Omega(n)) 
            = 
            \frac{1}{\log N} \sum_{n \leq N} \frac{1}{n} \, \mathbbm{1}_{I^c_{N,n}} (\Omega(n)) + O(\delta + \varepsilon),
        \]
        so that we can replace \( \mathbbm{1}_{I^c_N} \) by \( \mathbbm{1}_{I^c_{N,n}} \) in Equation \eqref{eqn: LOG1}.

        By the Hardy-Ramanujan Theorem, there is some \( N_0 \in \N \) such that for \( N \geq N_0 \). 
        \begin{equation}
            \frac{1}{N} \sum_{n \leq N} \mathbbm{1}_{I^c_{N,n}} (\Omega(n)) < \varepsilon. 
        \end{equation}
        Combined with partial summation, this implies for \( N > N_0\), 
        \begin{align*}
            \frac{1}{\log N} \sum_{n \leq N} \frac{1}{n} \, \mathbbm{1}_{I^c_{N,n}} (\Omega(n)) &= 
            \frac{1}{\log N} \Big( \frac{1}{N} \sum_{n \leq N} \mathbbm{1}_{I^c_{N,n}}(\Omega(n)) \Big)  
            \\
            & \quad \quad \quad \quad \quad + \frac{1}{\log N}\sum_{m \leq N-1} \frac{1}{m+1} \Big(\frac{1}{m} \sum_{n \leq m} \mathbbm{1}_{I^c_{N,n}} (\Omega(n)) \Big) \Big]
            \\
            & <
            \frac{\varepsilon}{\log N}\Big( 1 + \sum_{m = N_0}^N \frac{1}{m} \Big) + \frac{C_\varepsilon}{\log N}
            \\
            & \leq 
            \varepsilon + \frac{C_\varepsilon}{\log N}
        \end{align*}
        where \( C_\varepsilon\) is a constant not depending on \(N\). Taking \( N \) large enough so that \( C_\varepsilon \leq \varepsilon \log N \), we are done. 
        
    \end{proof}

    \begin{proof}[Proof of \cref{thm: Log Sweeping Out}]
        Let \( (X, \mu, T )\) be a non-atomic ergodic system and \( \varepsilon > 0\). By Equation \eqref{eqn: LogLog Weighted Expression},
        \[
            \Elog f(T^{\Omega(n)}x) 
            = 
            \underset{k\in [N]}{\mathbb{A}^{\xi_N}} f(T^k x)
        \]
        for any function \( f \) and \( x \in X \). \cref{lem: logarithmic weight support} guarantees that for all \( A \in \mathcal{B}\), 
        \begin{equation}
        \label{eqn: LOG5}
            \limsup_{N \to \infty} \Big| \frac{1}{\log N} \sum_{k \notin I_N} \xi_N(k) \mathbbm{1}_A(T^k x) \Big| \leq \limsup_{N \to \infty} \frac{1}{N} \sum_{k \notin I_N} \xi_N(k) \leq \varepsilon. 
        \end{equation}
        Additionally, by \cref{lem: Erdos Logarithmic Approx,lem: Erdos piN Approx}, 
        \[
            \frac{\xi_N(k)}{\log N} = \frac{1}{\log N}\frac{(\log\log N)^k}{k!} ( 1 + o(1)) = \frac{\pi_N(k)}{N} (1+o(1)),
        \]
        where both estimates are uniform for \( k \in I_N\). Hence by \cref{lem: Erdos to Gaussian},
        \begin{equation}
        \label{eqn: LOG6}
            \sum_{k \in I_N} \frac{\xi_N(k)}{\log N} \mathbbm{1}_A(T^k x) 
            = 
            \sum_{k \in I_N} \frac{1}{\sqrt{2\pi \log \log N}} \sum_{k \in I_N} e^{- \frac{1}{2}(\frac{k- \log \log N}{\sqrt{\log \log N}})^2} \mathbbm{1}_A(T^k x) + o(1). 
        \end{equation}
        Then \cref{prop: Operator SSO} yields the result.  
    \end{proof}

\section{Proof of \cref{thm: Perturbations of SSO}}
\label{sec: Proof Perturbations SSO}

    \begin{prop}\label{prop: SSO Interval Condition}
        Let \((a_n)\) be a sequence of positive integers. Suppose that for every \( N_0 \in \N \), \( C>0 \), and \( \varepsilon > 0 \), there exist integers \( r \) and \( (N_i)_{i\in[r]} \) satisfying \( N_i \geq N_0 \ \) for all \( i\in [r] \) and a sequence of intervals \( (J_i)_{i\in [r]} \) which has the following properties:
        \begin{enumerate}
            \item \( \dfrac{\#\{{n\in [N_i]: a_n\in J_i}\}}{N_i}\geq (1-\varepsilon)\) for every \( i\in [r] \).
            \item \( \dfrac{|\cup_{i\in [r]}J_i|}{\max_{i\in [r]}|J_i|}>C \).
        \end{enumerate}
        Then \( (a_n) \) has the strong sweeping out property.
    \end{prop}

    \begin{remark}
        One can prove results similar to \cref{prop: SSO Interval Condition} for aperiodic flows sampled along a real sequence \((a_n)\). To achieve such a result, one can apply \cite[Theorem 2.3]{Mondal2023}.
    \end{remark}
    
    \begin{proof}
         By \cref{lem: Transference}, it is enough to disprove a maximal ergodic inequality for a chosen system. Take the periodic system of period \( 2L \) endowed with the shift map. We want to show that for every \( N_0 \in \N \), \( D >0 \), and \( \varepsilon>0 \), there exists an interval \( [-L,L] \) and a set \( E \sse [-L,L] \) such that
        \begin{equation} 
        \label{eq:3.63}
            \# \Big\{k\in [-L,L] \,:\, \sup_{ N \geq N_0 } \EC{} \mathbbm{1}_E (k+a_n) > 1-\varepsilon\} 
            >
            D |E|.
        \end{equation}
        Take \( C = D/2 \) and choose \( (N_i)_{i\in [r]} \) and \( (J_i)_{i\in [r]} \) satisfying the hypotheses of the proposition. Let \( M = \max_{i\in [r]}{|J_i|} \). Choose \( L \) large enough so that \( [-L,L] \) contains \( J_i \) and \( -J_i \) for all \( i\in [r] \). Define \( E= [-M,M] \). Let \( k\in -J_i \) for some \( i\in [r] \). By the first property, 
        \[
            \dfrac{\#\{{n\in [N_i]: a_n\in J_i}\}}{N_i}
            \geq 
            (1-\varepsilon).
        \]
        This implies that
        \[
            \ECC{n}{[N_i]}\mathbbm{1}_E(k+{a_n})
            = 
            \dfrac{\#\{{n\in [N_i]: k+a_n\in E}\}}{N_i}
            \geq 
            (1-\varepsilon).
        \]
        Hence
        \begin{align*}
            \#\Big\{ k \in [-L,L] \,:\, \sup_{N\geq N_0}\EC{} \mathbbm{1}_E (k+a_n) > 1-\varepsilon \Big\}
            =
            \Big|\bigcup_{i\in [r]}J_i \Big| > C M = D |E| .
        \end{align*}
         where the last inequality follows from the second property in the statement of the proposition.
    \end{proof}

    \begin{proof}[Proof of \cref{thm: Perturbations of SSO}]
        Let \( \varepsilon > 0 \) and take \( E \) to be a fixed set satisfying 
        \[
            \lim_{N \to \infty} \frac{|E \cap [N]|}{N} > 1 - \varepsilon. 
        \]
        Define \( p_{n,E} := \max_{j \in E \cap [N]} |a_j - b_j|\). Suppose that \((a_n)\), \((b_n)\), and \((p_{n,E})\) satisfy the conditions of \cref{thm: Perturbations of SSO}. Let \( N_0 \in \N \), \( C > 0 \). We check that \( (a_n) \) satisfies the conditions of \cref{prop: SSO Interval Condition}. First, take \( R \in \N \) such that \( 1/R < \varepsilon \) and let \( K_0 \) and \( r \) be large positive integers to be chosen later. For \( i \in [r] \), define \( N_i := R^{(K_0+i)}\) and 
        \[
            J_i= [b_{N_{i-1}}-p_{N_r, E} \,,\, b_{N_i}+p_{N_r, E}].
        \]
        For now, choose \(K_0 \geq N_0\) large enough such that \( |E \cap [N]| > (1-\varepsilon) N \) for all \( N \geq N_1 \). Observe that for each \( n \in \{ N_{i-1},N_i\} \cap E \), we have \( a_n\in J_i \) by definition of \( p_{N_r, E} \). Hence,
        
        \begin{equation}\label{eq:K1}
            \dfrac{\# \{ n\in [N_i] \cap E : a_n \in J_i \} }{N_i}
            \geq 
            \frac{N_i-N_{i-1} - \varepsilon N_i}{N_i} 
            \geq 
            1- \varepsilon - \frac{1}{R}
            \geq 
            1 - 2\varepsilon.
        \end{equation}
        Hence the intervals \( J_i \) satisfy the first condition of \cref{prop: SSO Interval Condition}. To verify the second condition, it will be sufficient to show that for each \( i \in [r] \),
        
        \begin{equation}\label{eq:condition2}
            \dfrac{|\cup_{k \in [r]} J_k|}{\max_{k \in [r]} {J_k}} 
            \geq 
            \dfrac{b_{N_r}-b_{N_1}}{b_{N_i}-b_{N_{i-1}}+2p_{N_r, E}}
            \geq 
            C.
        \end{equation}

        The first inequality is apparent since 
        \[
            |\cup_{k \in [r]} J_k| = b_{N_r}-b_{N_1} + 2 p_{N_r, E} \geq b_{N_r}-b_{N_1}. 
        \]
        To prove the second, define an auxiliary sequence \( (s_k) \) by \( b_{R^{k}} := R^{k s_k} \). By the first hypothesis of \cref{thm: Perturbations of SSO}, the sequence \( s_k \) must tend to zero in \(k \). Else \( s_k \geq \delta \) for some fixed \( \delta > 0\) so that
        \[
            \frac{b_{R^{k}}}{(R^k)^\delta} = \frac{R^{k s_k}}{(R^k)^\delta} \geq 1.  
        \]
        We now show that 
        \begin{equation}\label{eq:3.64}
            \lim_{k \to \infty} \dfrac{b_{R^k}}{b_{R^k}- b_{R^{k-1}}} = \infty.  
        \end{equation}
        To the contrary, suppose \( 0<\dfrac{b_{R^k}}{b_{R^k} - b_{R^{k-1}}}< M \) for some fixed \( M > 0 \). Then
        \begin{align*}
            R^{\Big((k-1) s_{k-1}-k s_k \Big)} 
            \leq 
            1-\frac{1}{M}.
        \end{align*}
        Since \( R > 1 \), this implies \( k s_k -(k-1) s_{k-1} \geq \eta >0 \) for some fixed \( \eta > 0 \). But then \( ks_k \geq \eta k \), contradicting that \( (s_k)\) tend to zero.
        
        Then, by Equation \eqref{eq:3.64}, we can take \( K_0 \) large enough such that for  each \( i\in [r]\), 
        
        \begin{equation}
        \label{eq:choice1}
            \dfrac{b_{N_r}}{b_{N_i}-b_{N_{i-1}}}
            \geq 
            \dfrac{b_{N_i}}{b_{N_i}-b_{N_{i-1}}}
            \geq 
            4C.
        \end{equation}
        
        Now, take \( K_0 \) possibly larger so that it satisfies Equations \eqref{eq:K1} and \eqref{eq:choice1}, as well as
        \begin{equation}\label{eq:lowerbound}
            \dfrac{b_{K}}{p_{K,E}}\geq 8C 
        \end{equation}
        for all \( K \geq K_0 \). Finally, choose \( r \) large enough such that
        
        \begin{equation}\label{eq:choice2}
            \dfrac{b_{N_r}-b_{N_1}}{b_{N_r}}\geq \frac{1}{2}.
        \end{equation}
        
        Then, combining Equations \eqref{eq:choice1}, \eqref{eq:lowerbound}, and \eqref{eq:choice2}, we obtain Equation \eqref{eq:condition2}, completing the proof. 
    \end{proof}

    \begin{proof}[Proof of \cref{cor: CLT SSO}]
        Set \( a_n = a(n) \) and \( b_n = b(n) \). The sequence \((b_n)\) satisfies the first the condition of \cref{thm: Perturbations of SSO} by assumption. Now, let \( \varepsilon > 0\). By assumption, there exists a function \(p(n)\) and \(N_0 \in \N\) such that for \( N \geq N_0\), 
        \begin{equation}
        \label{eqn: Density}
            \# \{ n \leq N \,:\, |a(n) - b(n)| > p(n) \} < \varepsilon N. 
        \end{equation}
        Set \( E = \N \setminus \{ n \in N \,:\, |a_n - b_n| > p(n) \} \). Equation \eqref{eqn: Density} implies that \( d(E) > 1- \varepsilon \) and
        \[
            p_{N, E} 
            := 
            \max_{n \in E \cap [N]} |a_n - b_n|
            \leq 
            \max_{n \in E \cap [N]} p(n)
            \leq 
            p(N). 
        \]
        Then 
        \[
            \lim_{N \to \infty} \frac{p_{N, E}}{b_N} \leq \lim_{N \to \infty} \frac{p(N)}{b(N)} = 0, 
        \]
        satisfying the second condition. 
    \end{proof}

    \begin{proof}[Proof of \cref{cor: Divisor and little omega}]
        For \( x \in \R \), let \( \Phi(x) \) denote the standard normal distribution
        \[
            \Phi(x) = \int_{-x}^x e^{-t^2/2} dt.
        \]
        In \cite{Billingsley74}, it is shown that \(\omega(n)\) and \( \log d(n) \) satisfy the following central limit theorems:
        \[
            \lim_{N \to \infty} \frac{1}{N} \#\Big\{ n \leq N \,:\, |\omega(n) - \log \log N| < x \sqrt{\log \log N} \Big\} = \Phi(x)
        \]
        and 
        \[
            \lim_{N \to \infty} \frac{1}{N} \#\Big\{ n \leq N \,:\, |\log d(n) - \log 2 \log \log N| < x \log 2 \sqrt{\log \log N} \Big\} = \Phi(x).
        \]
        Applying \cref{cor: CLT SSO} with \(a(n) = \omega(n) \), \(b(n) = \log \log n\), and \( p(n) = \sqrt{\log \log n}\), we obtain that \( \omega(n)\) has the strong sweeping out property. Similarly, applying \cref{cor: CLT SSO} with \(a(n) = \log d(n) \), \(b(n) = \log 2\log \log n\), and \( p(n) = \log 2 \sqrt{\log \log n}\), we obtain that \( \log d(n) \) has the strong sweeping out property.   
        
    \end{proof}

\bibliography{MainReferenceLibrary}{}
\bibliographystyle{halpha}

\end{document}